\documentclass[12pt,leqno]{amsart}
\usepackage{graphicx} 

\usepackage{amsfonts,amsthm,amsmath,mathtools,amssymb,comment,xcolor,enumitem,cleveref}
\usepackage{subcaption}
\usepackage{ascmac}

\usepackage{braket}
\usepackage{stmaryrd}

\usepackage{tikz}
\tikzstyle{v} = [circle, draw, inner sep=1pt, minimum size=3pt, fill=black]

\theoremstyle{plain}
\newtheorem{thm}{Theorem}[section]

\newtheorem{lem}[thm]{Lemma}
\newtheorem{cor}[thm]{Corollary}

\theoremstyle{definition}
\newtheorem{df}[thm]{Definition}
\newtheorem{rem}[thm]{Remark}

\newtheorem{conj}[thm]{Conjecture}

\newcommand{\CC}{\mathbb{C}}

\newcommand{\NN}{\mathbb{N}}

\DeclareMathOperator{\Hom}{Hom}

\title{The Kneser chromatic function distinguishes trees}
\author{Yusaku Nishimura}
\address{School of Fundamental Science and Engineering, Waseda University, Tokyo 169-8555, Japan}
\email{n2357y@ruri.waseda.jp}
\date{}

\begin{document}

\begin{abstract}
  R.P. Stanley defined a invariant for graphs called the chromatic symmetric function and conjectured it is complete invariant for trees.
  Miezaki et.al generalised the chromatic symmetric function and defined the Kneser chromatic functions denoted by $\{X_{K_{\NN,k}}\}_{k\in\NN}$, and rephrase Stanley's conjecture that $X_{K_{\NN,1}}$ is a complete invariant for trees.
  This paper shows $X_{K_{\NN,2}}$ is a complete invariant for trees.
\end{abstract}

\maketitle

\section{Introduction}

Proper coloring in a simple graph $G$ with vertex set $V(G)$ is a way of coloring $V(G)$ such that each pair of adjacent vertices is colored with different colors.  
As an invariant derived from proper coloring, there exists a polynomial known as the chromatic polynomial.
The chromatic polynomial with indeterminate $x$, denoted by $\chi(G,x)$, is defined as the number of proper colorings of $G$ using at most $x$ colors.  
It is known that a proper coloring of $G$ can be viewed as a homomorphism from $G$ to a complete graph $K_n$, where $n$ is the number of colors.
Therefore, 
\[
  \chi(G,x)=|\Hom(G,K_x)|,
\]
where $\Hom(G,K_x)$ is the set of homomorphisms from $G$ to $K_x$.
Stanley generalized the chromatic polynomial and defined the chromatic symmetric function with indeterminates $\bold{x}$ as follows \cite{S95}:
\[
  X(G)\coloneq X(G,\bold{x})=\sum_{\phi\in \Hom(G,K_{\mathbb{N}})}\prod_{v\in V(G)}x_{\phi(v)}.
\]
Stanley conjectured that the chromatic symmetric function is a complete invariant for trees:
\begin{conj}[\cite{S95}]\label{conj:tree}
If $T_1$ and $T_2$ are non-isomorphic trees then 
\[
  X(T_1)\neq X(T_2). 
\] 
\end{conj}

This conjecture remains an open problem, and there are many works on this conjecture.  
For example, Heil and Ji showed that Stanley's conjecture holds for all trees with at most 29 vertices \cite{HJ19}.  
Martin, Morin, and Wagner proved that $X(T)$ completely includes the information of the degree sequence and path sequence of a tree $T$, which are invariants of trees \cite{MMW08}.
Several studies have been conducted on the generalization of the chromatic symmetric function.  
Hasebe and Tsujie generalized the chromatic symmetric function for posets, calling it the strict order quasisymmetric function, and proved that it is a complete invariant for rooted trees \cite{HT17}.  
Miezaki, Munemasa, Nishimura, Sakuma, and Tsujie defined a new invariant called the Kneser-chromatic function and showed that this invariant is a complete invariant for all finite graphs \cite{mmnst2024}.

\begin{df}[Kneser chromatic function \cite{mmnst2024}]
  Let $K_{\NN,k}$ be a Kneser graph, whose vertices are $k$-subsets of natural numbers and two vertices adjacent when intersection of them are empty.
  We define $x_u(u\in K_{\NN,k})$ as indeterminates and 
  \[
    X_{K_{\NN,k}}(G) :=
    X_{K_{\NN,k}}(G, x) :=
    \sum_{\varphi \in \Hom(G,K_{\NN,k})}
    \prod_{v\in V(G)}
    x_{\varphi(v)}. 
  \]
\end{df}

\begin{thm}[\cite{mmnst2024}]
\[
\{X_{K_{\NN,k}}(\bullet)\}_{k\in \NN}
\]
is a complete invariant for finite graphs. 
\end{thm}

Since $K_{\NN,1}=K_{\NN}$, the Kneser-chromatic function with $k=1$ is equal to the chromatic symmetric function.
Miezaki et al. also defined a new invariant called the Kneser-functional index.
\begin{df}[\cite{mmnst2024}]
  Let $\mathcal{G}$ be the set of all simple and finite graphs.
  For $\mathcal{A}\subset\mathcal{G}$, we define the Kneser-functional index $I_{K_{\NN}}(\mathcal{A})$ as
  \[
    I_{K_{\NN}}(\mathcal{A})=\min{\{t\in\NN\mid \text{$\{X_{K_\NN,k}\}_{k=1}^t$ is a complete invariant for $\mathcal{A}$ }\}}.
  \]
\end{df}
They rephrase Stanley's conjecture as $I_{K_{\NN}}(\mathcal{T})=1$, where $\mathcal{T}$ is the set of all finite tree graphs, and consider a natural question: the upper bound of $I_{K_{\NN}}(\mathcal{T})$.
In this paper, we prove the following theorem, specifically showing that the upper bound of $I_{K_{\NN}}(\mathcal{T})$ is $2$.

\begin{thm}\label{thm:main}
  $X_{K_{\NN,2}}(\bullet)$ is a conplete invariant for trees.
\end{thm}

\begin{cor}\label{cor:main}
  \[
    I_{K_\NN}(\mathcal{T})\leq 2
  \]
\end{cor}

\Cref{thm:main} is obtained from the following theorem:
\begin{thm}\label{lem:distTree}
  Let $\overline{\lambda} \in \tilde{\Lambda}_t$ and $\lambda \in \overline{\lambda}$.
  Define $G_{\lambda_0}$ as a tree which is obtained by removing any minimum leaf from $G_\lambda$.
  Then, $G_{\lambda_0}$ is isomorphic to $G$.
\end{thm}
Note that detailed definitions of certain terms in \Cref{lem:distTree} are provided in \Cref{sec:3}.

In \Cref{sec:2}, we introduce some notation and state the theorems that will be used in the proof of \Cref{lem:distTree}.  
In \Cref{sec:3}, we define a new invariant for trees and prove \Cref{lem:distTree}.  
In \Cref{sec:4}, we present the proof of \Cref{thm:main}.

\section{Preliminary}\label{sec:2}

We denote the vertex set of graph $G$ as $V(G)$ and the edge set as $E(G)$, and define $N_G(v)$ as the neighborhood of vertex $v$ in $G$. 
Additionally, let $d_v = |N_G(v)|$, and $d_G(u, v)$ represent the distance between vertices $u$ and $v$.

There is a useful theorem called the power sum expansion for the Kneser chromatic function. 
To introduce this, we define some terms.

\begin{df}[\cite{mmnst2024}]
  Define an equivalence relation $\sim$ between two multisets $\{I_{1}, \cdots, I_{n}\}$ and $\{J_1,\cdots,J_n\}$ consisting of $k$-subsets of $\NN$ by $\{I_{1}, \cdots, I_{n}\}\sim\{J_1,\cdots,J_n\}$ if there exists $\sigma\in S_\NN$ such that 
  \[
  \{I_{1}, \cdots, I_{n}\} = \{\sigma(J_{1}), \cdots, \sigma(J_{n})\},
  \]
  where $S_\NN$ is the symmetric group. The equivalence class of such multisets of size $n$ defined by this equivalence relation is denoted by $\mathcal{P}_{n}^{(k)}$.
  \end{df}

Since $X_{K_{\NN,k}}(G)$ is a member of the ring of formal power series $R_{K_{\NN,k}} \coloneqq \mathbb{C}\llbracket x_{w} \mid w \in V(K_{\NN,k}) \rrbracket$ and 
the automorphism group of the Kneser graph is isomorphic to $S_\NN$, $X_{K_{\NN,k}}(G)$ belongs to the invariant ring $R_{K_{\NN,k}}^{S_\NN}$.
Especially, because all monomials of $X_{K_{\NN,k}}(G)$ are degree $|V(G)|$, $X_{K_{\NN,k}}(G)\in Sym^{(k)}$, where $Sym^{(k)}$ is the subring of $R_{K_{\NN,k}}^{S_\NN}$ formed from finite degrees.
  
Let $\overline{\lambda}$ be any element in $\mathcal{P}_{n}^{(k)}$ and let $\lambda \in \overline{\lambda}$. 
Then, we can see $\lambda$ as an edge set of some $k$-uniform hyper-multigraph. 
We define the underlying set of $\lambda$ as
  \[
    \lambda_{\text{base}} \coloneq \bigcup_{S \in \lambda} S
  \]
and also define $G_\lambda$ as a hypergraph whose vertex set is $\lambda_{\text{base}}$ and hyperedge set is $\lambda$. From the definition of $\overline{\lambda}$, $G_\lambda$ does not depend on the choice of $\lambda$. 
Therefore, we denote $G_{\overline{\lambda}}$ as a $k$-uniform hyper-multigraph constructed by $\lambda \in \overline{\lambda}$.

We can partition $G_{\overline{\lambda}}$ into connected components $G_{\overline{\lambda}} = G_{\overline{\lambda}_1} \sqcup \cdots \sqcup G_{\overline{\lambda}_\ell}$. 
Then, define $p_{\overline{\lambda}} \in Sym^{(k)}$ by
\[
  p_{\overline{\lambda}} \coloneqq \prod_{i=1}^{\ell} \sum_{\{I_{1}, \dots, I_{n}\} \in \overline{\lambda}_i} x_{I_{1}} \cdots x_{I_{n}}.
\]
It is known that $p_{\overline{\lambda}}$ forms a linear basis for $Sym^{(k)}$ over $\CC$ \cite{mmnst2024}.

\begin{df}[\cite{mmnst2024}]
Let $G$ be a connected graph with $n$ vertices. 
We say that $\overline{\lambda} \in \mathcal{P}_{n}^{(k)}$ is \textit{admissible} by $G$ if there exists a bijection $\varphi \colon V(G) \to E(G_{\overline{\lambda}})$ such that for any $\{u,v\} \in E(G)$, it holds that $\varphi(u) \cap \varphi(v) \neq \emptyset$. 
We then define $\mathcal{A}^{(k)}_{G}$ as the set of elements in $\mathcal{P}_{n}^{(k)}$ that are admissible by $G$.

If $G$ is disconnected, and $G = G_{1} \sqcup \dots \sqcup G_{\ell}$ is its decomposition into connected components, we define $\mathcal{A}_{G}^{(k)} \coloneqq \mathcal{A}_{G_{1}}^{(k)} \times \dots \times \mathcal{A}_{G_{\ell}}^{(k)}$.
\end{df}
Note that if we say that $\lambda=\{I_{1}, \cdots, I_{n}\}$ is admissible by $G$, then it means that $\overline{\lambda}$ is admissible by $G$.

\begin{thm}[Power sum expansion \cite{mmnst2024}]\label{thm:pse}
  \[
    X_{K_{\NN,k}}(G) 
    = \sum_{S \subset E(G)}(-1)^{|S|}\sum_{\overline{\lambda} \in \mathcal{A}_{G_S}^{(k)}}p_{\overline{\lambda}},
  \]
  where $G_S$ is a spanning edge subgraph of $G$ whose vertex set is $V(G)$ and edge set is $S$.
\end{thm}

  Let $G$ be a tree with $n$ vertices, where $n\geq 2$.
  Since $p_{\overline{\lambda}}$ forms a linear basis, $X_{K_{\mathbb{N},2}}(G)$ can be uniquely represented as follows:
  \[
    X_{K_{\mathbb{N},2}}(G)=\sum_{\overline{\lambda}\in\Lambda_G}a_{\overline{\lambda}} p_{\overline{\lambda}},
  \]
  where $\Lambda_G\subset\mathcal{P}_n^{(2)}$.
  From \Cref{thm:pse}, we obtain 
  \[
    \Lambda_G=\bigcup_{S\subset E(G)}\mathcal{A}_{G_S}^{(2)}.
  \]
  Therefore, for any element $\overline{\lambda} \in \Lambda_G$, there exists a subset $S \subset E(G)$ such that $\overline{\lambda} \in \mathcal{A}_{G_S}$.
  In particular, if $G_{\overline{\lambda}}$ is connected, then such an $S$ also satisfies that $G_S$ is connected.
  Since $G$ is a tree, $G_S$ is connected if and only if $S = E(G)$.
  Hence, if $G_{\overline{\lambda}}$ is connected, then $\overline{\lambda} \in \mathcal{A}_G^{(2)}$.
  Then, define $\Lambda_t$ as follows:
  \[
    \Lambda_t\coloneq\{\overline{\lambda}\in\Lambda_G\mid \text{$G_{\overline{\lambda}}$ is a tree}\}.
  \]
  Since $G_{\overline{\lambda}}$ is connected, $\Lambda_t\subset\mathcal{A}_G^{(2)}$. 
  Therefore, for all $\overline{\lambda}\in \Lambda_t$, $\overline{\lambda}$ is admissible by $G$.
  
  We confirm some simple properties of $\Lambda_t$.  
  Let $\overline{\lambda}$ be any element of $\Lambda_t$.  
  Since $\Lambda_t \subset \mathcal{A}_G^{(2)}$ and $G_{\overline{\lambda}}$ does not have multiedges, $G_{\overline{\lambda}}$ has $n$ edges.  
  This implies $|V(G_{\overline{\lambda}})| = n + 1$.  
  Hereafter, we assume that $\lambda$ is always chosen as the representative element of $\overline{\lambda}$ such that $\lambda_{\text{base}}$ is $\{0,1,\ldots,n\}$.
  Since $\lambda$ is admissible by $G$, there exists a bijection from $V(G)$ to $\lambda$.  
  Let us denote this bijection by $\phi_\lambda$.
  
\section{Definition of a minimum degree sequence of a tree and proof of \Cref{lem:distTree}}\label{sec:3}

In this section, we define some terms for \Cref{lem:distTree}, and finally, we prove \Cref{lem:distTree}.

  \begin{df}[Rooted vertex sequence]
    A vertex sequence $\{a_i\}_{i=1}^n$ of a tree $T$ is called a \textit{rooted vertex sequence} with $v$ as the root if $a_1 = v$ and for any integers $i$ and $j$ satisfying $1 \leq i < j \leq n$, $d_T(v, a_i) \leq d_T(v, a_j)$.
  \end{df}
    
  If $\{a_i\}_{i=1}^n$ is a rooted vertex sequence with $v$ as the root, then for any vertex $a_i \neq v$, there exists exactly one vertex $a_{i_p}$ such that $a_{i_p}$ is adjacent to $a_i$ and $i_p < i$. 
  We call such a vertex $a_{i_p}$ the parent of $a_i$ with $v$ as the root.

  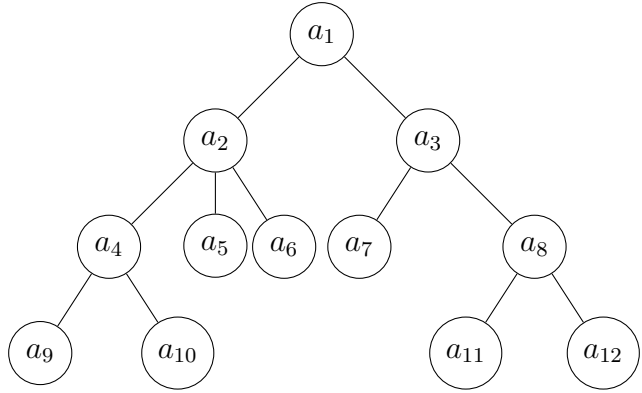
\begin{figure}
    \begin{center}
      \begin{tikzpicture}[scale=1, transform shape, node distance=2cm and 3cm, every node/.style={circle, draw}]
        \node (n1) at (0,0) {$a_1$};
    
        \node (n2) [below left of=n1] {$a_2$};
        \node (n3) [below right of=n1] {$a_3$};
    
        \node (n4) [below left of=n2] {$a_4$};
        \node (n5) [below right of=n2, xshift=-0.5cm] {$a_6$};
        \node (n6) [below left of=n3, xshift=0.5cm] {$a_7$};
        \node (n7) [below right of=n3] {$a_8$};
        \node (n8) [below of=n2, yshift=0.6cm] {$a_5$};
    
        
        \node (n9) [below right of=n4,xshift=-0.5cm] {$a_{10}$};
        \node (n10) [below left of=n4,xshift=0.5cm] {$a_9$};
        \node (n11) [below left of=n7,xshift=0.5cm] {$a_{11}$};
        \node (n12) [below right of=n7,xshift=-0.5cm] {$a_{12}$};
    
        \draw (n1) -- (n2);
        \draw (n1) -- (n3);
        \draw (n2) -- (n4);
        \draw (n2) -- (n5);
        \draw (n3) -- (n6);
        \draw (n3) -- (n7);
        \draw (n2) -- (n8);
        \draw (n4) -- (n9);
        \draw (n4) -- (n10);
        \draw (n7) -- (n11);
        \draw (n7) -- (n12);
    
      \end{tikzpicture} 
      \caption{Example of a rooted vertex sequence with $a_0$ as the root}\label{figure:case1}  
    \end{center}
    
  \end{figure}

  For example, in \Cref{figure:case1}, $\{a_i\}_{i=1}^{12}$ is a rooted vertex sequence with $a_1$ as the root.
  In this case, $a_2$ is a parent of $a_4$, $a_5$, and $a_6$ with $a_1$ as the root.
  
  Then, we define a new invariant of a tree.
  
  \begin{df}[Minimum rooted vertex sequence and minimum degree sequence]
    Let $T$ be a tree and $\{a_i\}_{i=1}^n$ be a rooted vertex sequence with $v$ as the root. 
    $\{a_i\}_{i=1}^n$ is called a \textit{minimum rooted vertex sequence} if any rooted vertex sequence $\{a'_i\}_{i=1}^n$ with $v$ as the root satisfies
    \[
      (d_v, d_{a_2}, \cdots, d_{a_n}) \leq_{lex} (d_v, d_{a'_2}, \cdots, d_{a'_{n}}),
    \]
    where $\leq_{lex}$ denotes the lexicographical order. 
    We denote $r(T,v)$ as the degree sequence of the minimum rooted vertex sequence with $v$ as the root. Specifically, a vertex $v$ is said to be a \textit{minimum leaf} if for all vertices $u \in V(T)$, $r(T,v) \leq_{lex} r(T,u)$ holds. 
    Additionally, $r(T,v)$ is called a \textit{minimum degree sequence} of $T$ if $v$ is a minimum leaf.
  \end{df}
  
  Note that, generally, a minimum leaf is not unique, but the minimum degree sequence of a tree is always unique. 
  We represent $r(T,v)_k$ as the degree of the $k$th vertex in $r(T,v)$. 
  We denote $r(T)$ as a minimum degree sequence of $T$. 
  The definition of $r(T)_k$ is the same.
  
  For example, in \Cref{figure:case1}, 
  \[
    (a_{11},a_{8},a_{12},a_3,a_7,a_1,a_2,a_5,a_6,a_4,a_9,a_{10})
  \]
  is one of the minimum rooted vertex sequences with $a_{11}$ as the root and 
  \[
    r(T,a_{11})=(1,3,1,3,1,2,4,1,1,3,1,1).
  \]
  Since $r(T,a_{11})$ is the minimum for all vertices in the tree, $a_{11}$ is a minimum leaf and $r(T)=r(T,a_{11})$.

  Then, define $\tilde{\Lambda}_t$ as follows:
  \[
  \tilde{\Lambda}_t \coloneq \{\overline{\lambda} \in \Lambda_t \mid \text{$r(G_{\overline{\lambda}})$ is the minimum in $\Lambda_t$ with respect to lexicographical order}\}.
  \]
  Then, we will prove \Cref{lem:distTree}.
  Before the proof, we prepare the following lemma.  

\begin{lem}\label{lem:minDA}
  For any $\overline{\lambda}\in\tilde{\Lambda}_t$,
  \[
    r(G_{\overline{\lambda}})\leq_{lex} (1,2,r(G)_2,r(G)_3,\cdots,r(G)_n).
  \]
\end{lem}

\begin{proof}
  Let $\{a_i\}_{i=1}^n$ be a minimum rooted vertex sequence of $G$ with $a_1$ as the root and $a_1$ as a minimum leaf.  
  From the minimality of $\tilde{\Lambda}_t$, it is sufficient to prove the existence of $\overline{\lambda'}$ such that $\overline{\lambda'}$ is in $\Lambda_t$ and $r(G_{\overline{\lambda'}})\leq_{lex}(1,2,r(G)_2,r(G)_3,\cdots,r(G)_n)$.  
  We consider constructing such a $\lambda'$.
  
  Let $a_{i_p}$ be a parent of $a_i$ with $a_1$ as the root, and we define $i_p=0$.
  Then, we define $\phi: V(G) \longrightarrow \begin{pmatrix} \NN \cup \{0\} \\ 2 \end{pmatrix}$ as follows:
  \[
  \phi(a_i) = \{i_p, i\}.
  \]
  An example of $\phi$ is shown in \Cref{figure:case4}.
  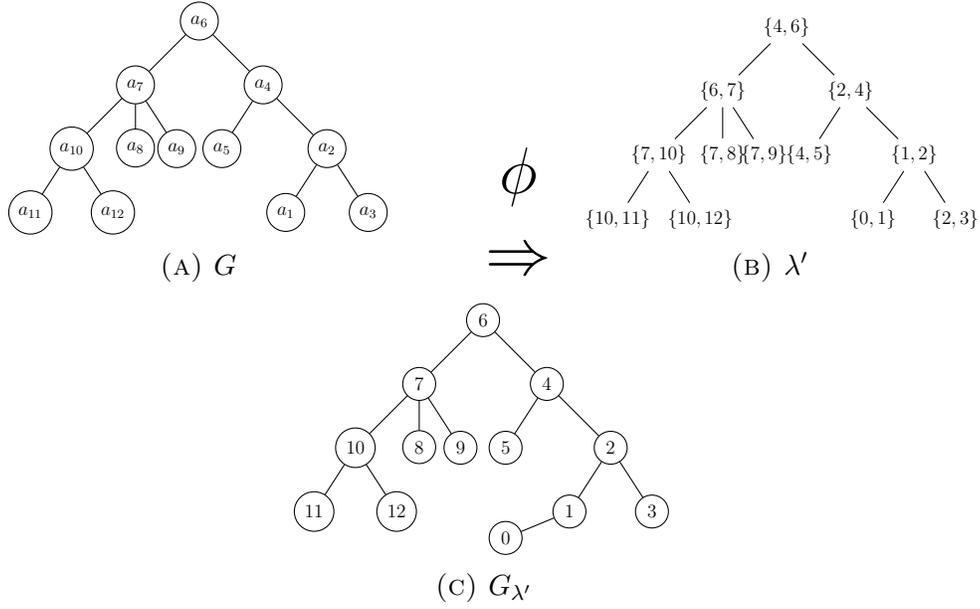
\begin{figure}
    \begin{minipage}[b]{0.40\columnwidth}
      \begin{tikzpicture}[scale=0.6, transform shape, node distance=2cm and 3cm, every node/.style={circle, draw}]
        \node (n1) at (0,0) {$a_6$};
    
        \node (n2) [below left of=n1] {$a_7$};
        \node (n3) [below right of=n1] {$a_4$};
    
        \node (n4) [below left of=n2] {$a_{10}$};
        \node (n5) [below right of=n2, xshift=-0.5cm] {$a_9$};
        \node (n6) [below left of=n3, xshift=0.5cm] {$a_5$};
        \node (n7) [below right of=n3] {$a_2$};
        \node (n8) [below of=n2, yshift=0.6cm] {$a_8$};
    
        
        \node (n9) [below right of=n4,xshift=-0.5cm] {$a_{12}$};
        \node (n10) [below left of=n4,xshift=0.5cm] {$a_{11}$};
        \node (n11) [below left of=n7,xshift=0.5cm] {$a_{1}$};
        \node (n12) [below right of=n7,xshift=-0.5cm] {$a_{3}$};
    
        \draw (n1) -- (n2);
        \draw (n1) -- (n3);
        \draw (n2) -- (n4);
        \draw (n2) -- (n5);
        \draw (n3) -- (n6);
        \draw (n3) -- (n7);
        \draw (n2) -- (n8);
        \draw (n4) -- (n9);
        \draw (n4) -- (n10);
        \draw (n7) -- (n11);
        \draw (n7) -- (n12);
    
      \end{tikzpicture}  
      \subcaption{$G$}
    \end{minipage}
    \hfill
    \hspace{0.05\columnwidth}
    \begin{tikzpicture}
        \node(arrow) at (0,0){\huge $\Rightarrow$}; 
        \node at (arrow.north) [above=3mm] {\huge $\phi$};
    \end{tikzpicture}%
    \hfill
    \begin{minipage}[b]{0.40\columnwidth}
      \begin{tikzpicture}[scale=0.6, transform shape, node distance=2cm and 3cm]
        \node (n1) at (0,0) {$\{4,6\}$};
    
        \node (n2) [below left of=n1] {$\{6,7\}$};
        \node (n3) [below right of=n1] {$\{2,4\}$};
    
        \node (n4) [below left of=n2] {$\{7,10\}$};
        \node (n5) [below right of=n2, xshift=-0.5cm] {$\{7,9\}$};
        \node (n6) [below left of=n3, xshift=0.5cm] {$\{4,5\}$};
        \node (n7) [below right of=n3] {$\{1,2\}$};
        \node (n8) [below of=n2, yshift=0.6cm] {$\{7,8\}$};
    
        
        \node (n9) [below right of=n4,xshift=-0.5cm] {$\{10,12\}$};
        \node (n10) [below left of=n4,xshift=0.5cm] {$\{10,11\}$};
        \node (n11) [below left of=n7,xshift=0.5cm] {$\{0,1\}$};
        \node (n12) [below right of=n7,xshift=-0.5cm] {$\{2,3\}$};
    
        \draw (n1) -- (n2);
        \draw (n1) -- (n3);
        \draw (n2) -- (n4);
        \draw (n2) -- (n5);
        \draw (n3) -- (n6);
        \draw (n3) -- (n7);
        \draw (n2) -- (n8);
        \draw (n4) -- (n9);
        \draw (n4) -- (n10);
        \draw (n7) -- (n11);
        \draw (n7) -- (n12);
    
      \end{tikzpicture}  
      \subcaption{$\lambda'$}
 
    \end{minipage}
    \begin{center}
      \begin{minipage}[b]{0.40\columnwidth}
        \begin{tikzpicture}[scale=0.6, transform shape, node distance=2cm and 3cm, every node/.style={circle, draw}]
          \node (n1) at (0,0) {$6$};
      
          \node (n2) [below left of=n1] {$7$};
          \node (n3) [below right of=n1] {$4$};
      
          \node (n4) [below left of=n2] {$10$};
          \node (n5) [below right of=n2, xshift=-0.5cm] {$9$};
          \node (n6) [below left of=n3, xshift=0.5cm] {$5$};
          \node (n7) [below right of=n3] {$2$};
          \node (n8) [below of=n2, yshift=0.6cm] {$8$};
      
          
          \node (n9) [below right of=n4,xshift=-0.5cm] {$12$};
          \node (n10) [below left of=n4,xshift=0.5cm] {$11$};
          \node (n11) [below left of=n7,xshift=0.5cm] {$1$};
          \node (n12) [below right of=n7,xshift=-0.5cm] {$3$};
          \node (n0) [below of=n6] {$0$};
      
          \draw (n1) -- (n2);
          \draw (n1) -- (n3);
          \draw (n2) -- (n4);
          \draw (n2) -- (n5);
          \draw (n3) -- (n6);
          \draw (n3) -- (n7);
          \draw (n2) -- (n8);
          \draw (n4) -- (n9);
          \draw (n4) -- (n10);
          \draw (n7) -- (n11);
          \draw (n7) -- (n12);
          \draw (n11) -- (n0);
      
        \end{tikzpicture}  
        \subcaption{$G_{\lambda'}$}
      \end{minipage}  
    \end{center} 
    \caption{An example of $\phi$}\label{figure:case4}
  \end{figure}
  We check that $Im(\phi)=\lambda'$ has the desired properties.  
  First, we show the admissibility of $\lambda'$.  
  For any pair of vertices $a_i$ and $a_j$ in $G$, if they are adjacent, then $i_p=j$ or $j_p=i$.  
  This implies $\{i_p,i\} \cap \{j_p,j\} \neq \emptyset$.  
  Therefore, $\lambda'$ is admissible for $G$.  
  Next, we show that $G_{\lambda'}$ is a tree.  
  Let $G_{\lambda'_0}$ be an induced subgraph of $G_{\lambda'}$ with vertex $0$ removed.  
  In fact, $G_{\lambda'_0}$ is isomorphic to $G$ because $E(G)=\{a_{i_p},a_{i}\}_{i=2}^n$ and $E(G_{\lambda'_0})=\phi(V(G) \setminus \{a_1\})$.
  Therefore, $G_{\lambda'_0}$ is a tree.  
  Since $G_{\lambda'}$ is obtained by adding a leaf $0$ to $G_{\lambda'_0}$, $G_{\lambda'}$ is also a tree.  
  Finally, we check $r(G_{\lambda'}) \leq_{lex} (1,2,r(G)_2,r(G)_3,\cdots,r(G)_n)$.  
  Because $G_{\lambda'_0}$ is isomorphic to $G$ and $a_1$ is a minimum leaf of $G$, vertex $1$ in $G_{\lambda'_0}$ is also a minimum leaf and  
  \[
    r(G_{\lambda'_0},1) = r(G).
  \]  
  Additionally, since $0$ is a leaf adjacent to $1$, $r(G_{\lambda'},0) = (1,2,r(G)_2,r(G)_3,\cdots,r(G)_n)$.  
  Then,  
  \[
    r(G_{\lambda'}) \leq_{lex} r(G_{\lambda'},0) = (1,2,r(G)_2,r(G)_3,\cdots,r(G)_n).
  \]  
  From the above, there exists $\lambda'$ such that $\overline{\lambda'} \in \Lambda_t$ and $r(G_{\lambda'})$ satisfies the desired inequality.  
  Hence, we obtain $r(G_{\overline{\lambda}}) \leq_{lex} (1,2,r(G)_2,r(G)_3,\cdots,r(G)_n)$.
\end{proof}

Then, we prove \Cref{lem:distTree}.

\begin{proof}[Proof of \Cref{lem:distTree}]
  
  First, we construct a bijection from $V(G)$ to $V(G_{\lambda_0})$ by using a rooted vertex sequence in $G_\lambda$.
  Let $\overline{\lambda}\in\Lambda_t$, $\lambda\in\overline{\lambda}$, and $l$ be a minimum leaf of $G_\lambda$.
  Define $\{a_i\}_{i=0}^n$ as any minimum rooted vertex sequence in $G_\lambda$ with $l$ as the root and define $\sigma\in S_{n+1}$ such that $\sigma(\{a_i\}_{i=0}^n)=\{i\}_{i=0}^n$.
  Then, $\sigma(\lambda)$ is also in $\overline{\lambda}$ and $\{i\}_{i=0}^n$ is a minimum rooted vertex sequence in $G_{\sigma(\lambda)}$.
  Therefore, without loss of generality, we assume that $\{i\}_{i=0}^n$ is a rooted vertex sequence with $0$ as the root, and $0$ is the minimum leaf in $G_\lambda$.
  In this case, for any vertex $v$ in $G$, if $\phi_\lambda(v)=\{x_1,y_1\}$ and $x_1<y_1$, then a parent of $y_1$ is $x_1$ with $0$ as the root in $G_\lambda$.
  Especially, for any vertices $u\neq v$ in $G$, if $\phi_\lambda(u)=\{x_2,y_2\}$ and $x_2<y_2$, then $y_1\neq y_2$.
  This is because the parent of any vertex is unique.
  Therefore, 
  \[
    \max{\phi_\lambda(u)} \neq \max{\phi_\lambda(v)}.
  \]
  From this property, we obtain the bijection $\tau$ which is defined as follows:
  \begin{align*}
    \tau: \lambda & \longrightarrow \{1, \cdots, n\} \\
            \{x, y\} & \mapsto \max{(x, y)}.
  \end{align*}
  Hence, $\tau \circ \phi_\lambda$ is also a bijection from $V(G)$ to $V(G_{\lambda_0})$, where $G_{\lambda_0}$ is the graph $G_\lambda$ minus vertex $0$.
  We denote $\tau_\lambda = \tau \circ \phi_\lambda$.  
  For example, in \Cref{figure:case4}, $a_6 \in V(G)$ corresponds to $\{4,6\} \in \lambda'$ via $\phi_{\lambda'}$  
  and $\{4,6\}$ corresponds to $6 \in V(G_{\lambda'_0})$ via $\tau$.  
  Therefore, $a_6$ corresponds to $6$ via $\tau_{\lambda'}$.  
  Hereinafter, $\tau$ and $\tau_\lambda$, corresponding to a rooted vertex sequence $\{i\}_{i=0}^n$ of $G_\lambda$, are defined above.  
  Using this bijection $\tau_\lambda$, we will prove that $\tau_\lambda^{-1}$ is a isomorphism.

  Let $\tau_\lambda^{-1}(i) = v_i$ for all vertices $i$ in $V(G_{\lambda_0})$. 
  We show $\tau_\lambda^{-1}(N_{G_{\lambda_0}}(i)) = N_G(v_i)$ by induction on $i$. 
  First, we prove the base case $i = 1$. 
  From \Cref{lem:minDA}, we obtain $r(G_\lambda,0)_0 = 1$ and $r(G_\lambda,0)_1 \leq 2$. 
  If $r(G_\lambda, 0)_1 = 1$, then the degree of vertex $1$ in $G_\lambda$ is $1$, and $G_\lambda$ must be a path with $2$ vertices. 
  This contradicts the assumption that $G_\lambda$ has at least $3$ vertices.
  Therefore, $r(G_\lambda,0)_1 = 2$, and this implies 
  \begin{align}
      \{\{x,y\} \in \lambda \mid 0 \in \{x,y\}\} &= \{\{0,1\}\} \notag\\
      \{\{x,y\} \in \lambda \mid 1 \in \{x,y\}\} &= \{\{0,1\}, \{1,2\}\} \notag.
  \end{align}
  Hence, we obtain $\phi_{\lambda}(v_1)=\{0,1\}$ and the edge set of $G_{\lambda_0}$ is $\lambda \setminus \{\{0,1\}\}$; it implies $N_{G_{\lambda_0}}(1)=\{2\}$. 
  Since $\lambda$ is admissible,
  \begin{align*}
    \phi_\lambda(N_G(v_1))&\subset \{\{x,y\}\in\lambda\mid \{0,1\}\cap\{x,y\}\neq\emptyset\}\\
    &=\{\{x,y\}\in\lambda\mid 0\in\{x,y\}\}\cup\{\{x,y\}\in\lambda\mid 1\in\{x,y\}\}\\
    &=\{\{0,1\},\{1,2\}\}.
  \end{align*}
  Since $v_1 \notin N_G(v_1)$,
  \begin{align*}
    \phi_\lambda(N_G(v_1)) & \subset \{\{1,2\}\}.
  \end{align*}
  Apply $\tau$ to both sides of this inclusion,
  \[
    \tau_\lambda(N_G(v_1)) \subset \{2\}.
  \]
  Additionally, since $G$ is a tree, $|N_G(v_1)| \geq 1$.
  Therefore,
  \[
    \tau_\lambda(N_G(v_1)) = \{2\} = N_{G_{\lambda_0}}(1)
  \]
  and we get $\tau_\lambda^{-1}(N_{G_{\lambda_0}}(1)) = N_G(v_1)$.

  Next, we assume that $\tau_\lambda^{-1}(N_{G_{\lambda_0}}(i))=N_G(v_i)$ holds for all $i$ satisfying $1 \leq i \leq m<n$.
  Then, there exists $p_{m+1} < m+1$ such that $\phi_\lambda(v_{m+1}) = \{p_{m+1}, m+1\}$.
  Obviously, $m+1 \in N_{G_\lambda}(p_{m+1})$. 
  Since $m+1 \geq 2$, we get $p_{m+1} \neq 0$ and 
  \begin{equation}\label{eq:1}\tag{I}
      m+1 \in N_{G_{\lambda_0}}(p_{m+1}).
  \end{equation}
  Similarly to before, since $\lambda$ is admissible,
  \begin{align*}
      \phi_\lambda(N_G(v_{m+1})) & \subset \{\{x,y\} \in \lambda \mid p_{m+1} \in \{x,y\}\} \cup \{\{x,y\} \in \lambda \mid m+1 \in \{x,y\}\}.
  \end{align*}
  Additionally, 
  \begin{align*}
    \tau(\{\{x,y\}\in\lambda \mid p_{m+1} \in \{x,y\}\}) &\subset N_{G_{\lambda_0}}(p_{m+1}) \cup \{p_{m+1}\}, \\
    \tau(\{\{x,y\}\in\lambda \mid m+1 \in \{x,y\}\}) &\subset N_{G_{\lambda_0}}(m+1) \cup \{m+1\}.
  \end{align*}
  Therefore, 
  \[
    \tau_\lambda(N_G(v_{m+1})) \subset N_{G_{\lambda_0}}(p_{m+1}) \cup N_{G_{\lambda_0}}(m+1).
  \]
  Let $v_{p_{m+1}} = \tau_\lambda^{-1}(p_{m+1})$. 
  From the assumption, $\tau_\lambda^{-1}(N_{G_{\lambda_0}}(p_{m+1})) = N_G(v_{p_{m+1}})$. 
  Here, from \Cref{eq:1}, 
  \[
    v_{m+1} = \tau_\lambda^{-1}(m+1) \in N_G(v_{p_{m+1}}).
  \]
  Then, since adjacent vertices in a tree do not have a common neighbor,
  \begin{align*}
    \tau_\lambda(N_G(v_{m+1})) \cap N_{G_{\lambda_0}}(p_{m+1}) &= \tau_\lambda(N_G(v_{m+1})) \cap \tau_\lambda(N_G(v_{p_{m+1}})) \\
    &= \tau_\lambda(N_G(v_{m+1}) \cap N_G(v_{p_{m+1}})) \\
    &= \emptyset.
  \end{align*}
  Hence,
  \begin{equation}\label{eq:2}\tag{II}
    \tau_\lambda(N_G(v_{m+1}))\subset N_{G_{\lambda_0}}(m+1).
  \end{equation}
  Because $N_{G_{\lambda_0}}(i) = N_{G_\lambda}(i)$ for all $i \geq 2$, 
  \[
    |N_{G_{\lambda_0}}(m+1)| = r(G_\lambda, 0)_{m+1}.
  \]
  From now on, we will show that $|N_{G}(v_{m+1})| = r(G_{\lambda}, 0)_{m+1}$.
  From the assumption that $|N_{G_{\lambda_0}}(i)| = |N_{G}(v_i)|$ holds for all $1 \leq i \leq m$, it follows that
  \begin{align*}
    r(G_\lambda,0) &= (1,2,|N_{G_{\lambda_0}}(2)|,|N_{G_{\lambda_0}}(3)|,\cdots,|N_{G_{\lambda_0}}(n)|) \\
    &= (1,2,|N_{G}(v_2)|,|N_{G}(v_3)|,\cdots,|N_{G}(v_n)|).
  \end{align*}
  Since $0$ is a minimum leaf of $G_\lambda$, $r(G_\lambda,0) = r(G_{\overline{\lambda}})$ and from \Cref{lem:minDA},
  \begin{equation}\label{eq:3}\tag{III}
      r(G_\lambda,0) \leq_{lex} (1,2,r(G)_2,r(G)_3,\cdots, r(G)_n).
  \end{equation}
  Suppose that there exist $v_{x_{m+2}}, v_{x_{m+3}}, \cdots, v_{x_{n}}$ such that $(v_1, v_2, \cdots, v_m, v_{m+1}, v_{x_{m+2}}, \cdots, v_{x_n})$ is a rooted vertex sequence.
  Then, from the minimality of $r(G)$, 
  \begin{equation}\label{eq:4}\tag{IV}
    \{|N_G(v_i)|\}_{i=1}^{m+1} \geq_{lex} \{r(G)_i\}_{i=1}^{m+1}.
  \end{equation}
  From \Cref{eq:3} and \Cref{eq:4}, we obtain 
  \[
    \{|N_G(v_i)|\}_{i=1}^{m+1}=\{r(G)_i\}_{i=1}^{m+1}
  \] 
  and especially $|N_{G}(v_{m+1})|=r(G_{\lambda},0)_{m+1}$.
  Therefore, to prove $|N_{G}(v_{m+1})| = r(G_{\lambda},0)_{m+1}$, we consider the existence of such a rooted vertex sequence. Let $k$ be the distance between vertex $1$ and vertex $m+1$ in $G_{\lambda_0}$, and define 
  \[
    D_{x}(1) \coloneq \{i \in V(G_{\lambda_0}) \mid d_{G_{\lambda_0}}(1,i) \leq x\},
  \]
  where $x$ is any integer. Since $\{i\}_{i=1}^n$ is a rooted vertex sequence in $G_{\lambda_0}$,
  \begin{align*}
    D_{k-1}(1) &\subset \{1,2,\cdots,m+1\} \subset D_k(1).
  \end{align*}
  Here, define $D_{x}(v_1)$ as follows:
  \[
    D_{x}(v_1) \coloneq \{v_i \in V(G) \mid d_{G}(v_1, v_i) \leq x\}.
  \]
  By the inductive hypothesis, 
  \begin{align*}
    D_{k-1}(v_1) &= \tau_\lambda(D_{k-1}(1)) \subset \{v_1, \cdots, v_m, v_{m+1}\} \subset D_{k}(v_1) = \tau_\lambda(D_{k}(1)).
  \end{align*}
  It follows that there exists a desired vertex sequence.
  From the above, $|N_{G}(v_{m+1})| = r(G_{\lambda},0)_{m+1}$. 
  Using $|N_{G_{\lambda_0}}(m+1)| = r(G_\lambda,0)_{m+1}$ and \Cref{eq:2},
  \[
      N_{G_{\lambda_0}}(m+1) = \tau_{\lambda}(N_G(v_{m+1})).
  \]
  
  Hence, by induction, for any vertex $i\in V(G_{\lambda_0})$,
  \[
    \tau_\lambda^{-1}(N_{G_{\lambda_0}}(i))=N_G(v_i).
  \]
  This means $\tau_\lambda^{-1}$ is a graph homomorphism, and it follows that $\tau_\lambda$ is an isomorphism.
\end{proof}

\begin{rem}
  Generally, a minimum leaf is not unique in a tree. Therefore, $G_{\lambda_0}$ seems to depend on the choice of a minimum leaf. Actually, we can say that either a minimum leaf of $G_{\lambda}$ is unique or $G_{\lambda}$ is a path. In the proof of \Cref{lem:distTree}, we show \Cref{eq:3} and \Cref{eq:4}, and from induction, we also get
  \begin{equation*}
    \{|N_G(v_i)|\}_{i=1}^{n} = \{r(G)_i\}_{i=1}^{n}.
  \end{equation*}
  It follows that vertex $1$ is also a minimum leaf in $G_{\lambda_0}$. Consequently, $G_{\lambda}$ is a tree obtained by adding leaf $0$ to a minimum leaf $1$ in $G$.
  Next, we show that $0$ is a minimum leaf in $G_{\lambda}$ and establish its uniqueness. For any other leaf $l$ in $G_{\lambda}$, since the degree of vertex $1$ is $1$ in $G_{\lambda_0}$ and $2$ in $G_{\lambda}$, we have
  \[
    r(G_{\lambda},l) >_{lex} r(G_{\lambda_0}).
  \]
  If $G_{\lambda}$ is not a path, then there exists $k \leq n$ such that $r(G_\lambda)_k \geq 3$ and for all $m$ smaller than $k$, $r(G_\lambda)_m \leq 2$.
  Therefore, 
  \begin{align*}
    r(G_{\lambda_0})&=(1,\overbrace{2,\cdots,2}^{k-1},r(G)_k,r(G)_{k+1},\cdots,r(G)_n),\\
    r(G_{\lambda},0)&=(1,\overbrace{2,\cdots,2}^{k},r(G)_k,r(G)_{k+1},\cdots,r(G)_n).
  \end{align*}
  This implies $r(G_{\lambda},0) <_{lex} r(G_{\lambda_0})$. 
  Hence, for all other leaves $l$, 
  \[
    r(G_{\lambda},0) <_{lex} r(G_{\lambda},l),
  \]
  which shows $0$ is a minimum leaf and it is unique. 
  When $G_{\lambda}$ is a path, then there are two minimum leaves, and no matter which vertex is removed, the resulting graph is isomorphic. 
  Furthermore, for each case, 
  \[
    r(G_{\overline{\lambda}})=r(G_{\lambda},0) = (1,2,r(G)_2,\cdots,r(G)_n).
  \] 
  Therefore, the equality in \Cref{lem:minDA} always holds.
\end{rem}

\section{Proof of \Cref{thm:main}}\label{sec:4}

In this section, we prove \Cref{thm:main}.

\begin{proof}[Proof of \Cref{thm:main}]
  Let $G_1$ and $G_2$ be trees, and assume that 
  \[
    X_{K_{\NN,2}}(G_1) = X_{K_{\NN,2}}(G_2) = \sum_{\overline{\lambda} \in \Lambda_G} a_{\overline{\lambda}} p_{\overline{\lambda}}.
  \]
  Then, we can uniquely determine $\tilde{\Lambda}_t$ from $\Lambda_G$, and by \Cref{lem:distTree}, we can construct $G_{\lambda_0}$ which is isomorphic to both $G_1$ and $G_2$. 
  Therefore, $G_1 \simeq G_2$, which implies that $X_{K_{\NN,2}}(\bullet)$ is a complete invariant.
\end{proof}


\begin{rem}
  In \cite{mmnst2024}, there is a question: is $\mathcal{A}^{(2)}_\bullet$ a complete invariant for trees?
  From the proof of \Cref{thm:main}, we only use the property that $\mathcal{A}_t \subset \mathcal{A}_G$. 
  Therefore, we can also conclude that $\mathcal{A}^{(2)}_\bullet$ is a complete invariant for trees.
\end{rem}
\begin{rem}
  Since $X_{K_{\NN,2}}(\bullet)$ has much more information than $\mathcal{A}^{(2)}_\bullet$, $X_{K_{\NN,2}}(\bullet)$ may distinguish a broader class than trees, such as unicyclic graphs or bipartite graphs.

  {Question:} 
  Does there exist a finite integer $m$ such that $I_{K_{\NN}}(\mathcal{G})=m$, where $\mathcal{G}$ is the set of all simple and finite graphs? 
  Alternatively, for any natural number $k$, does there exist a pair of graphs $G$ and $G'$ such that $X_{K_{\NN,k}}(G)=X_{K_{\NN,k}}(G')$ holds but they are not isomorphic to each other?

\end{rem}

\section*{Acknowledgements}
The author thanks Professor Tsuyoshi Miezaki and Master’s student Naoki Fuji for their helpful discussions and comments.


\end{document}